\documentclass[11 pt]{amsart}
\usepackage{amsfonts, amsmath, amsopn, amssymb, latexsym, graphics, graphicx, enumerate,hyperref}

\DeclareFontFamily{OT1}{rsfs}{}
\DeclareFontShape{OT1}{rsfs}{n}{it}{<-> rsfs10}{}
\DeclareMathAlphabet{\mathscr}{OT1}{rsfs}{n}{it}

\newcommand{\Z}{\mathbb{Z}}
\newcommand{\Q}{\mathbb{Q}}

\begin{document}

\newtheorem{thm}[subsection]{Theorem}
\newtheorem*{thm*}{Theorem}
\newtheorem{lem}[subsection]{Lemma}
\newtheorem{cor}[subsection]{Corollary}
\newtheorem*{cor*}{Corollary}
\newtheorem{prop}[subsection]{Proposition}

\theoremstyle{definition}
\newtheorem{lect}[subsection]{Lecture}
\newtheorem{rem}[subsubsection]{Remark}
\newtheorem{ex}[subsection]{Example}
\newtheorem{Def}[subsection]{Definition}
\newtheorem{exer}[subsection]{Exercise}
\newtheorem{conj}[subsection]{Conjecture}
\newtheorem{prob}[subsection]{Problem}
\newtheorem{ques}[subsection]{Question}

\title{Division Algebras With Infinite Genus}
\author{Jeffrey S. Meyer}

\maketitle

\section{Introduction}

\qquad 
\textit{To what extent does the maximal subfield spectrum of a division algebra determine the isomorphism class of that algebra?} It has been shown that over some fields a quaternion division algebra's isomorphism class is largely if not entirely determined by its maximal subfield spectrum.  However in this paper, we show that there are  fields for which the maximal subfield spectrum says little to nothing about a quaternion division algebra's isomorphism class.

\qquad Let $D$ be a division algebra with center $k$ and consider the collection of fields $E$ such that $k\subset E\subset D$.  The maximal elements, called \textbf{maximal subfields}, of this poset are always degree $\sqrt{\dim_k D}$ field extensions of $k$.  For each $D$, its \textbf{maximal subfield spectrum} is the set
$$\mathrm{Max}(D)=\left\{ \mbox{Isomorphism classes of maximal subfields of $D$}\right\}.$$
When two division algebras have the same maximal subfield spectrum, they are called \textbf{weakly isomorphic}.  In (\cite{PR}, 5.4) Prasad and Rapinchuk pose this following question: \textit{For what fields $k$ does weakly isomorphic imply isomorphic?}  They observe that over global fields, weakly isomorphic implies isomorphic.     However for even finitely generated fields, there are open questions.
In \cite{GS}, Garibaldi and Saltman produce an infinitely generated field over which there are two nonisomorphic weakly isomorphic quaternion division algebras.    With this in mind define \textbf{the genus of $D$} to be the set:
$$\mathrm{Gen}(D)=\left\{ \parbox{4.0in}{Isomorphism classes of algebras $D'$ weakly isomorphic to $D$}\right\}.$$
Chernousov, A. Rapinchuk, and I. Rapinchuk show in \cite{CRR} that over a finitely generated field, $\mathrm{Gen}(D)$ is finite.  
In this paper, we show that there are quaternion division algebras with infinite genus.  This is done by generalizing a construction of Garibaldi and Saltman in \cite{GS}.  However, it should be noted that the naive approach to generalizing this construction does not work and hence the result requires careful consideration. 

\vspace{0.5pc}
\textbf{Theorem A.} \textit{There exists a division algebra with infinite genus.}
\vspace{0.5pc}

This theorem can be rephrased in the language of algebraic groups.  Quaternion division algebras give rise to algebraic groups of type $A_1$ and maximal subfields give rise to maximal tori.  In the literature, algebraic $k$-groups with the same isomorphism classes of maximal tori are called \textbf{isotoral} or \textbf{weakly $k$-isomorphic}.  As a corollary to Theorem A, we prove the following.

\vspace{0.5pc}
\textbf{Theorem B.} 
\textit{There exist fields over which there are infinitely many nonisomorphic isotoral groups of type $A_1$. }
\vspace{0.5pc}

\qquad In section 4 we prove a natural strengthening of these results, and along the way introduce the notion of a \textbf{linking field extension}, which we hope will be of independent interest.  Using this we prove the following.

\vspace{0.5pc}
\textbf{Theorem C.} 
\textit{There exists a field $K$ over which
\begin{enumerate}
\item there are infinitely many nonisomorphic quaternion division algebras with center $K$, and   
\item any two quaternion division algebra with center $K$ are pairwise weakly isomorphic.
\end{enumerate}
}
\vspace{0.5pc}

In fact, there are infinitely many nonisomorphic fields which satisfy the conditions of Theorem C.


\section{Background on quaternion division algebras}

\qquad Let $k$ be a field not of characteristic 2.  In what follows, we use the following notation:
\begin{itemize}
\item $\mathrm{Quat}(k)$ is the set of isomorphism classes of quaternion algebras with center $k$
\item $\mathrm{Quat}^*(k)$ is the subset of $\mathrm{Quat}(k)$ consisting of the quaternion division algebras.  
\item If $k'/k$ is a field extension $\mathscr{E}_{k'/k}:\mathrm{Quat}(k) \to \mathrm{Quat}(k')$ is the extension of scalars map $D\mapsto D\otimes_k k'$.
\end{itemize}
Two division algebras over the same field are said to be \textbf{linked} if they share a common maximal subfield.  A field $k$ is \textbf{linked} if any two quaternion division algebras with center $k$ are linked. Local and global fields (i.e., $\Q$) are all linked (\cite{Lam} VI.3.6).   To any nonhyperbolic (regular) $r$-dimensional quadratic form $\varphi$ over $k$, we may define its \textbf{``big'' function field} to be $k[\varphi]:=\mathrm{Frac}(k[x_1, \ldots, x_r]/(\varphi))$ where $\varphi$ is expressed as a homogeneous polynomial of degree 2 in the indeterminate $x_i$, $1\le i\le r$.  An \textbf{$n$-fold Pfister form} over $k$ is a quadratic form of the form $\bigotimes_{i=1}^n\langle 1, a_i\rangle$ for $a_i\in k^\times$.  The norm form of a quaternion division algebra $D$ with center $k$ is an anisotropic $2$-fold Pfister form over $k$, which we denote $q_D=\langle 1,-a\rangle \otimes \langle 1, -b\rangle$.  The isomorphism class of $D$ is determined by $q_D$ and is represented by the \textbf{Hilbert symbol} $\left( \frac{a,b}{k}\right)$.  For more on these objects, see \cite{Lam}.

\qquad For the reader's convenience, we include the following two results on anisotropic forms.

\begin{prop}[\cite{Lam} Cor.~X.4.10]\label{anisotropicstable}Let 
\begin{enumerate}
\item $F$ be a field of characteristic not $2$, 
\item $\varphi$ be a $2^n$-dimensional quadratic form over $F$, $n\ge 1$, and
\item $q$ be an anisotropic $n$-fold Pfister form over $F$.  
\end{enumerate}Then the following are equivalent.
\begin{enumerate}
\item $q$ and $\varphi$ are Witt equivalent.  In particular if $\varphi$ is also anisotropic, they are isometric.
\item $q$ becomes isotropic over $F[\varphi]$ 
\end{enumerate}
\end{prop}

\begin{thm}[\cite{Lam} Thm~X.4.34]\label{hoffmann} (Hoffmann's Theorem) Let 
\begin{enumerate}
\item $F$ be a field of characteristic not $2$, 
\item $q$ be an anisotropic quadratic form over $F$,
\item $\varphi$ a quadratic form so that $F[\varphi]$ is defined, and 
\end{enumerate}
If there exists an integer $n$ such that $\dim q\le 2^n<\dim \varphi,$
 then $q$ remains anisotropic over $F[\varphi]$.
\end{thm}


\section{Construction of a quaternion division algebra with infinite genus}

\qquad In this section, we prove Theorem A, and as a consequence Theorem B.  We begin by proving Proposition \ref{pushingext} which is the main tool in proving Theorem A.

\begin{prop}
\label{pushingext}
Let 
\begin{enumerate}
\item $F$ be a field of characteristic not $2$, 
\item $\mathcal{D}\subset \mathrm{Quat}^*(F)$ be a set whose elements are pairwise linked,
\item $S \subset F^\times/(F^\times)^2$.
\end{enumerate}
Then there exists a field extension $F_S/F$ such that 
\begin{enumerate}
\item the restriction $\mathscr{E}_{F_S/F}|_{\mathcal{D}}:\mathcal{D}\to \mathrm{Quat}^*(F_S)$ is injective,
\item each algebra in the image $\mathscr{E}_{F_S/F}(\mathcal{D})$  contains $F_S(\sqrt{c})$ for all $c\in S$.
\end{enumerate}
\end{prop}

\begin{proof}
For each $D\in \mathcal{D}$, denote its norm form by $q_D= \langle 1, -a_D, -b_D, a_Db_D\rangle$, where $a_D, b_D\in F^\times$.  Recall $q_D$ is an anisotropic 2-fold Pfister form of discriminant 1.  For two algebras $D, D'\in \mathcal{D}$, let $q_{D,D'}$ be the 2-fold Pfister form which is similar to the anisotropic part of $q_D-q_{D'}$.  Note that $q_{D,D'}$ has discriminant 1.  For each $c\in S$, define $T_c=\{D\in \mathcal{D}\ | \ F(\sqrt{c}) \mbox{ is not a subfield of }D\}$.  Define the set of pairs $\mathscr{P}= \{(c,D) \ | \ \mbox{$c\in S$ and $D\in T_c$}\}$ and for each $p:=(c, D)\in \mathscr{P}$, define the quadratic form 
$$\varphi_{p}:=\varphi_{(c,D)}:=cW_p^2-a_DX_p^2 -b_DY_p^2+a_Db_DZ_p^2$$
where $W_p, X_p, Y_p, Z_p$ are indeterminants.  Note that this form has discriminant $c\neq 1$.  Furthermore, the existence of an isotropic vector for $\varphi_{(c,D)}$ would imply the existence of $x,y,z\in F$ for which $c=a_{D}x^2+by^2-a_{D}b_Dz^2$ which would contradict our assumption that $D\in T_c$.
Hence $\varphi_{(c,D)}$ is anisotropic.  
Since they have different discriminants, for no pair $(c,D)$ is $\varphi_{(c,D)}$ isometric to any $q_D$ or $q_{D,D'}$.  

Fix a well ordering $<$ on $\mathscr{P}$ and let $p_0$ denote its least element.  Define fields $F_{p_0}:=F[\varphi_{p_0}]$ and $F_p:=(\bigcup_{p'<p}F_{p'})[\varphi_p]$, $p\in \mathscr{P}$, and let $F_S:=\bigcup_{p\in \mathscr{P}}F_p$.  

\qquad We  show this is the desired field.  First note that $q_D$ and $q_{D,D'}$ remain anisotropic over $F_S$ since by Proposition \ref{anisotropicstable} they remain anisotropic over each $F_p$, $p\in \mathscr{P}$.  Note that
\begin{enumerate}
\item $q_D$ remaining anisotropic implies $D\otimes F_S$ is still a division algebra, i.e., the image $\mathscr{E}_{F_S/F}$ lies in $\mathrm{Quat}^*(F_S)$,
\item $q_{D,D'}$ remaining anisotropic implies $D\otimes F_S$ and $D'\otimes F_S$ remain nonisomorphic over $F_S$, and hence the restriction $\mathscr{E}_{F_S/F}|_{\mathcal{D}}$ is injective.
\end{enumerate}

Next, pick any two $D,D'\in \mathcal{D}$ and $c\in S$ such that $F(\sqrt{c})$ is a subfield of $D'$ but not $D$, and denote $p=(c,D)$.  Then $\varphi_p$ is isotropic over $F_p$ and hence $c=a_{D}(X_p/W_p)^2+b_D(Y_p/W_p)^2-a_{D}b_D(Z_p/W_p)^2$.  It follows that $F_p(\sqrt{c})$ embeds in $D\otimes F_p$, and hence $F_S(\sqrt{c})$ embeds in $D\otimes F_S$.
\end{proof}

\qquad Proposition \ref{pushingext} creates a field $F_S$ over which all the division algebras of $\mathscr{E}_{F_S/F}(\mathcal{D})$  have the same quadratic extensions \textit{coming from $F$}.  However in general,  the algebras of $\mathscr{E}_{F_S/F}(\mathcal{D})$ are not weakly isomorphic because there are now new quadratic extensions coming from $F_S$ which could lie in one of the extended algebras but not in one of the others.



\begin{thm}\label{theorem1}Let 
\begin{enumerate}
\item $k$ be a field of characteristic not $2$, 
\item $\mathcal{D}\subset \mathrm{Quat}^*(k)$ be a set whose elements are pairwise linked,
\end{enumerate}
Then there exists a field extension $K/k$ such that 
\begin{enumerate}
\item the restriction $\mathscr{E}_{K/k}|_{\mathcal{D}}:\mathcal{D}\to \mathrm{Quat}^*(K)$ is injective, and  
\item the algebras in the image $\mathscr{E}_{K/k}(\mathcal{D})$ are pairwise weakly isomorphic.
\end{enumerate}
\end{thm}

\begin{proof} We construct $K$ with a tower of extensions produced in Proposition \ref{pushingext}.
Define the field $k_0=k$ and then recursively define $k_i, i\in \Z_{>0}$  to be $F_S$ after applying the theorem to $F=k_{i-1}$ for $$S=\{c\in k^\times/(k^\times)^2 \ | \mbox{  There exist $D,D'\in \mathcal{D}$ such that $F(\sqrt{c})$ is a maximal subfield of $D$ but not $D'$}\}.$$  Let $K=\bigcup_{i\ge 0} k_i$.  We now show this is the desired field.  First, by iterating the theorem, the image of $\mathscr{E}_{K/k}|_{\mathcal{D}}$ is injective into $\mathrm{Quat}^*(K)$.  Second, let $D, D'\in \mathcal{D}$ and $E$ be a maximal subfield of $D\otimes K$.  Then $E=K(\sqrt{a})$ where $a\in K$, and hence, there is some finite $n\ge 0$ such that $a\in k_n$.  Therefore $k_{n}(\sqrt{a})$ lies in $D'\otimes k_{n+1}$ and hence $E$ lies in $D'\otimes K$.
\end{proof}

\qquad For later convenience, let $P(k)$ denote the field obtained by applying Theorem \ref{theorem1} to $k$ and $\mathcal{D}=\mathrm{Quat}^*(k)$.  Let $k$ be a field with infinitely many nonisomorphic linked quaternion division algebras.  For example, $k$ could be a global field.  Then $P(k)$  has a division algebra of infinite genus, proving Theorem A.  Noting that anisotropic groups of type $A_1$ are obtained by looking at the norm 1 group of division algebras Theorem B follows.  

%
%

\section{Linked Algebras and Linking Extensions}

\qquad Having now proven Theorem A, it is natural to pose the following question: is there a field $K$ such that $\mathrm{Quat}^*(K)$ is infinite and for which \textit{any} two elements are pairwise weakly isomorphic?  Unfortunately, we cannot simply iterate the procedure in the proof of theorem A since the procedure only works for linked algebras.  Even if you start the above process with a linked field $k$, the first field in the tower $k_1$ may no longer be linked, hence we can only perform the procedure to the image of $\mathrm{Quat}^*(k)$ and not all of $\mathrm{Quat}^*(k_1)$.  In this section, we generalize our construction procedure to account for this issue.
We define a \textbf{linking extension} of $k$ to be a field extension $k'/k$ such that all division algebras in the image of $$\mathscr{E}_{k'/k}:\mathrm{Quat}^*(k)\to \mathrm{Quat}(k')$$
are pairwise linked.  There always exists a linking extension, namely the algebraic closure $\overline{k}$ of $k$.  However the image of the extension map $\mathscr{E}_{k'/k}$ might be finite even if $\mathcal{D}$ is infinite.

\begin{prop}\label{linkingext}
Let $F$ be a field of characteristic not 2.  There exists an extension $L/F$ for which
\begin{enumerate}
\item $L$ is a linking extension, and
\item Linked nonisometric algebras in $\mathrm{Quat}^*(F)$ remain nonisometric division algebras after extension.
\end{enumerate}
\end{prop}

\begin{proof}
As above, for each $D\in \mathcal{D}$, denote its norm form by $q_D$.  For two linked algebras $D, D'\in \mathcal{D}$, let $q_{D,D'}$ be the 2-fold Pfister form which is similar to the anisotropic part of $q_D-q_{D'}$.   For two unlinked algebras $D, D'\in \mathcal{D}$, let $\varphi_{D,D'}$ be the 6-dimensional form which is the anisotropic part of $q_D-q_{D'}$.  Let $\mathscr{P}$ be the set of pairs of unlinked algebras and let $<$ be a well-ordering of $\mathscr{P}$ with least element $(D_0, D_0')$.  Define fields $F_{(D_0,D_0')}:=F[\varphi_{(D_0,D_0')}]$ and $F_p:=(\bigcup_{p'<p}F_{p'})[\varphi_p]$, $p\in \mathscr{P}$ and let $L:=\bigcup_{p\in \mathscr{P}}F_p$. 

\qquad We now show this is the desired field.  First note that any pair $p=(D,D')\in \mathscr{P}$, they become linked in $F_p$, hence in $L$ and thus $L$ is a linking extension.  Next note that $q_D$ and $q_{D,D'}$ remain anisotropic over $L$ since by Hoffmann's Theorem (Theorem \ref{hoffmann}) they remain anisotropic over each $F_p$, $p\in \mathscr{P}$.
\end{proof}

\begin{prop}
Let $k$ be a field of characteristic not 2.  There exists an extension $L(k)/k$ for which
\begin{enumerate}
\item $L(k)$ is linked, and
\item Linked nonisometric algebras in $\mathrm{Quat}^*(k)$ remain nonisometric division algebras after extension.
\end{enumerate}
\end{prop}

\begin{proof}  We construct $L(k)$ with a tower of linking extensions.
Define the field $k_0=k$ and then recursively define $k_i, i\in \Z_{>0}$,  to be $L$ obtained after applying proposition \ref{linkingext} to $F=k_{i-1}$.  Let $L(k)=\bigcup_{i\ge 0} k_i$.  We now show this is the desired field.  First, by iterating the proposition, linked nonisometric algebras remain nonisometric division algebras after extension.  Second, let $D, D'\in \mathrm{Quat}^*(L(k))$.  Then,  using Hilbert symbols (see Section 2 and \cite{Lam}), there exists $a,b,a',b'\in L$ such that $D=\left(\frac{a,b}{L(k)}\right)$ and $D'=\left(\frac{a',b'}{L(k)}\right)$.  Hence there exists some finite $n\ge 0$ such that $a,b,a',b'\in k_n$ and hence over $k_{n+1}$, $D=\left(\frac{a,b}{k_{n+1}}\right)$ and $D'=\left(\frac{a',b'}{k_{n+1}}\right)$ are linked.  Since being linked is stable under extension, it follows $D$ and $D'$ are linked.
\end{proof}

Putting together this procedure and the procedure of section 3 yields Theorem C.  
\begin{thm}There exists a field $K$ for which
\begin{enumerate}
\item $\mathrm{Quat}^*(K)$ is infinite, and   
\item the algebras in $\mathrm{Quat}^*(K)$ are pairwise weakly isomorphic.
\end{enumerate}
\end{thm}

\begin{proof}
Let $k$ be a field with infinitely many nonisomorphic linked quaternion division algebras.  For example, $k$ could be a number field.  Set $k_0=k$ and then recursively define $k_i, i\in \Z_{>0}$  to be $P(L(k_{i-1}))$.  Let $K=\bigcup_{i\ge 0} k_i$.  I now show this is the desired field.  First, at each step, the image of nonisomorphic linked division algebras remain nonlinked division algebras, and hence by the initial assumption on $k$,  $\mathrm{Quat}^*(K)$ is infinite.   Second, let $D, D'\in \mathrm{Quat}(K)$, $D\neq D'$, and let $E$ be a maximal subfield of $D$.  Then $E=K(\sqrt{c})$ where $c\in K$.  Furthermore $D'=\left(\frac{a,b}{K}\right)$ for some $a, b\in K$.   Hence, there is some finite $n\ge 0$ such that $a,b,c\in k_n$ and hence $k_{n}(\sqrt{c})$ lies in $\left(\frac{a,b}{k_{n+1}}\right)$ and hence $E$ lies in $D'$.
\end{proof}

\begin{cor}
There exist infinitely many nonisomorphic fields which satisfy properties (1) and (2).
\end{cor}
\begin{proof}For each odd prime $p$, let $k_p$ be a global field of characteristic $p$.  Let $K_p$ be the field obtained by performing to $k_p$ the procedure in the proof of the theorem.  Then each $K_p$ satisfies properties (1) and (2), and since any two are of different characteristic, they are not isomorphic.
\end{proof}

\section{Acknowledgments} 
\qquad We thank Andrei Rapinchuk for suggesting this problem and Matthew Stover for countless helpful conversations.  The author was supported by the NSF RTG grant 1045119.


\begin{thebibliography}{BoHC}

\bibitem[CRR]{CRR}V.I. Chernousov, A.S. Rapinchuk, and I.A. Rapinchuk
\newblock\emph{On the genus of a division algebra}. 
\newblock C. R. Math. Acad. Sci. Paris 350, \textbf{17-18}, 807Ð812, (2012).

\bibitem[Ga]{G}S. Garibaldi
\newblock\emph{Outer automorphisms of algebraic groups and determining groups by their maximal tori}.
\newblock  Michigan Mathematical Journal, \textbf{61}, \#2 (2012), 227-237.

\bibitem[GS]{GS}S. Garibaldi and D. Saltman
\newblock\emph{Quaternion algebras with the same subfields}.
\newblock Quadratic forms, linear algebraic groups, and cohomology: Developments in Mathematics, \textbf{18} 
Springer, (2010), 225-238.

\bibitem[Lam]{Lam}T.Y. Lam
\newblock\emph{Introduction To Quadratic Forms Over Fields}.
\newblock  Graduate Studies in Mathematics , \textbf{67}, AMS, Providence, RI (2005).

\bibitem[PR]{PR}G. Prasad and A. S. Rapinchuk
\newblock\emph{Weakly commensurable arithmetic groups and
isospectral locally symmetric spaces}.
\newblock  Inst. Hautes ƒtudes Sci. Publ. Math. \textbf{109} (2009), 113Ð184.




\end{thebibliography}
\end{document}